\documentclass[11pt]{amsart}

\newtheorem{theorem}{Theorem}[section]
\newtheorem{lemma}[theorem]{Lemma}

\newtheorem{corollary}[theorem]{Corollary}

\theoremstyle{definition}
\newtheorem{definition}[theorem]{Definition}

\theoremstyle{remark}

\numberwithin{equation}{section}

\begin{document}

\title[$f$-Orthomorphisms and $f$-linear operators on the order dual]
 {$f$-Orthomorphisms and $f$-linear operators on the order dual of
an $f$-algebra}
\author[Y. feng]
{Ying Feng }
\address{Department of Mathematics, Southwest Jiaotong University,
Chengdu 610031, P.R. China}
 \email{fengying@home.swjtu.edu.cn}
\author[J.X. Chen]
{Jin Xi Chen}
\address{Department of Mathematics, Southwest Jiaotong
University, Chengdu 610031, P.R. China}
\email{jinxichen@home.swjtu.edu.cn}
\author[Z.L. Chen]
{Zi Li Chen}
\address{Department of Mathematics, Southwest Jiaotong
University, Chengdu 610031, P.R. China}
\email{zlchen@home.swjtu.edu.cn}

\thanks{The authors were supported in part by the Fundamental Research Funds for the Central Universities (SWJTU11CX154, SWJTU12ZT13)}



\begin{abstract}
In this paper we consider the $f$-orthomorphisms and $f$-linear
operators on the order dual of an $f$-algebra. In particular, when
the $f$-algebra has the factorization property (\,not necessarily
unital\,), we prove that the orthomorphisms, $f$-orthomorphisms and
$f$-linear operators on the order dual are precisely the same class
of operators.

\end{abstract}

\maketitle \baselineskip 5mm

\section{Introduction}

\par Let $A$ be an $f$-algebra with $^{\circ}(A^{\sim})=\{0\}$. Recall that we can define a multiplication
 on $(A^{\sim})^{\sim}_{n}$, the order continuous part of the order bidual
 of $A$, with respect to which $(A^{\sim})^{\sim}_{n}$ can also be made an
 $f$-algebra. This is done in three steps:
\begin{enumerate}
                                         \item $A\times A^{\sim}\rightarrow A^{\sim}$
\\$(a,f)\longmapsto f\cdot a: (f\cdot a)(b)=f(ab)$ for $b\in A$,

                                         \item $(A^{\sim})^{\sim}_{n}\times A^{\sim}\rightarrow A^{\sim}$
\\$(F,f)\longmapsto F\cdot f: (F\cdot f)(a)=F(f\cdot a)$ for $a\in A$,
                                         \item $(A^{\sim})^{\sim}_{n}\times (A^{\sim})^{\sim}_{n}\rightarrow (A^{\sim})^{\sim}_{n}$
\\$(F,G)\longmapsto F\cdot G: (F\cdot G)(f)=F(G\cdot f)$ for $f\in  A^{\sim}$.
\end{enumerate}
With the so-called Arens multiplication defined in Step (3)
$(A^{\sim})^{\sim}_{n}$ is an Archimedean (and hence commutative)
$f$-algebra. Moreover, If $A$ has a multiplicative unit, then
$(A^{\sim})^{\sim}_{n}=(A^{\sim})^{\sim}$, the whole  order bidual
of $A$.  The mapping $V:(A^{\sim})^{\sim}_{n}\rightarrow
\textrm{Orth}(A^{\sim})$ defined by $V(F)=V_{F}$ for all $F\in
(A^{\sim})^{\sim}_{n}$, where $V_{F}(f)=F\cdot f$ for every $f\in
A^{\sim}$, is an algebra and Riesz isomorphism. See \cite {Huijsmans
and de Pagter 1, B. Turan} for details.
\par Let $A$ be an $f$-algebra. A Riesz space $L$ with $^{\circ}(L^{\sim})=\{0\}$ is said to be
a (left) \textit{$f$-module} over $A$ (cf. \cite{LP, B. Turan}) if
$L$ is a left module over $A$, and satisfies the following two
conditions:

(i) for each $a\in A^{+}$ and $x\in L^{+}$ we have $ax\in L^{+}$,

(ii) if $x\perp y$, then for each $a\in A$ we have $a\cdot x \perp y
$.

When $A$ is an $f$-algebra with unit $e$, saying $L$ is a unital
$f$-module over $A$  implies that the left multiplication satisfies
$e\cdot x=x$ for all $x\in L$. From Corollary 2.3 in \cite {B.
Turan} we know that if $L$ is an
 $f$-module over $A$, then $L^{\sim}$ is an $f$-module over $A$ (and
 ($A^{\sim})^{\sim}_{n}$). The $f$-module $L$ over $A$ with unit $e$ is said to be \textit{topologically full} with respect to $A$ if for  two arbitrary vectors $x$, $y$ satisfying $0\leq y\leq x $ in
$L$, there exists a net $0\leq a_{\alpha}\leq e$ in
 $A$ such that $a_{\alpha}\cdot x\rightarrow y$ in
 $\sigma(L,\,L^{\sim})$. If $L$ is topologically full with respect
 to $A$, then $L^{\sim}$ is topologically full with respect
 to $(A^{\sim})^{\sim}_{n}$ \cite[Proposition 3.12]{B. Turan}.
 \par Let $A$ be a unital $f$-algebra and $L$, $M$ be $f$-modules over $A$. $T\in L_{b}(L,\,M)$ is called an
 \textit{$f$-linear operator} if $T(a\cdot x)=a\cdot Tx$ for each $a\in A$
 and $x\in L$. The collection of all $f$-linear operators will be denoted by $L_{b}(L, M; A)$. For each $x\in L$ and $f\in L^{\sim}$, we can define $\psi_{x,\,f}\in
 A^{\sim}$  by $\psi_{x,\,f}(a)=f(a\cdot x)$ for all $a\in
 A$. Let $S(x):=\{\psi_{x,\,f}:\,f\in L^{\sim}\}$. Then $S(x)$ is an
 order ideal in $A^{\sim}$ \cite {B. Turan}. $T\in L_{b}(L,\,M)$ is said to be an
 \textit{$f$-orthomorphism} if $S(Tx)\subseteq S(x)$ for each $x\in
 L$. The collection of all $f$-orthomorphisms will be denoted by
 Orth$(L, M; A)$. Turan \cite {B. Turan} showed that  $\textrm{Orth}(L, M; A)=L_{b}(L, M; A)$ whenever $M$ is
 topologically full with respect to $A$.
\par Clearly, $A^{\sim}$ is an $f$-module over the $f$-algebras $A$ and $(A^{\sim})^{\sim}_{n}$, respectively. If $A$ is unital , then
$A$ is topologically full with respect to itself (\cite[Proposition
2.6]{B. Turan}). From the above remarks we know that $A^{\sim}$ is
topologically full with respect to $(A^{\sim})^{\sim}_{n}$, and
hence the $f$-orthomorphisms and $f$-linear operators
 are precisely the same class of operators, that is,
$$ \textrm{Orth}(A^{\sim}, A^{\sim};(A^{\sim})^{\sim}_{n}
)=L_{b}(A^{\sim}, A^{\sim};(A^{\sim})^{\sim}_{n}).\eqno{(\ast)}$$
\par An $f$-algebra $A$ is said to be \textit{square-root
closed} whenever for any $a\in A$ there exists $b\in A$ such that
$|\,a|=b^{\,2}$. An immediate example is that a uniformly complete
$f$-algebra with unit element is square-root closed \cite{Huijsmans
de Pagter}. However, a square-root closed $f$-algebra is not
necessarily unital. For instance, $c_{\,0}$, with the familiar
coordinatewise operations and ordering, is a square-root closed
$f$-algebra without unit.  We recall  that an Archimedean
$f$-algebra $A$ is said to have the \textit{factorization property}
if, given $a\in A$, there exist $b,c\in A$ such that $a=bc$. It
should be noted that if $A$ is unital or square-root closed , then
$A$ has the factorization property.
\par In this paper, we do not have to assume that the $f$-algebras are
unital. We modify the definition of the $f$-orthomorphism introduced
by Turan \cite[Definition 3.7]{B. Turan} and consider the
$f$-orthomorphisms and $f$-linear operators on the order dual of an
$f$-algebra. In particular, when the $f$-algebra with separating
order dual has the factorization property, we prove that the
orthomorphisms, $f$-orthomorphisms and $f$-linear operators on the
order dual are precisely the same class of operators, that is, the
above equality $(\ast)$ still holds.
\par Our notions are standard. For the theory of Riesz spaces, positive operators and
$f$-algebras, we refer  the reader to the monographs \cite{AB, M,
Zaanen}.

\section{$f$-orthomorphisms on the order dual}
\par Let $A$ be an $f$-algebra with separating order dual (and hence $A$ Archimedean!) and $f\in A^{\sim}$. We  consider the mapping
$T_{f}:(A^{\sim})^{\sim}_{n}\rightarrow A^{\sim}$ defined by
$T_{f}(F)=F\cdot f$ for all $F\in (A^{\sim})^{\sim}_{n}$. It should
be noted that the mapping $V:(A^{\sim})^{\sim}_{n}\rightarrow
\textrm{Orth}(A^{\sim})$ defined by $V(F)=V_{F}$ for all $F\in
(A^{\sim})^{\sim}_{n}$, where $V_{F}(f)=F\cdot f$ for every $f\in
A^{\sim}$, is an algebra and Riesz isomorphism ( cf. \cite
[Proposition 2.2]{B. Turan}).
\begin{theorem}\label{Theorem 2.1}
 For $0\leq f\in A^{\sim}$, $T_{f}$ is an interval preserving lattice homomorphism.
\end{theorem}

\begin{proof}
 Clearly, $T_{f}$ is linear and positive. Since the mapping $V$ is a lattice homomorphism and $V_{F}
,\,V_{G}\in \textrm{Orth}(A^{\sim})$ for $F,\,G\in
(A^{\sim})^{\sim}_{n}$, we have
\begin{eqnarray*}
T_{f}(F\vee G)=(F\vee G)\cdot f&=&V_{F\vee G}(f)\\&=&(V(F\vee
G))(f)\\&=&(V(F)\vee V(G))(f)\\
&=&(V(F)(f))\vee (V(G)(f))\\&=&F\cdot f\vee G\cdot
f\\&=&T_{f}(F)\vee T_{f}(G).
\end{eqnarray*}
Hence $T_{f}$ is a lattice homomorphism.
 \par Next we show that $T_{f}$ is
an interval preserving operator. We identify $x$ with its canonical
image $x''$ in $(A^{\sim})^{\sim}_{n}$ and denote the  restriction
of $T_{f}$ to $A$ by $T_f|_{A}$. Then
$$T_{f}|_{A}(x)=T_{f}(x'')=x''\cdot f=f\cdot x.$$Thus for each $F\in (A^{\sim})^{\sim}_{n}$ and $x\in
A$, we see that
$$((T_{f}|_{A})'(F))(x)=F((T_{f}|_{A})(x))=F(f\cdot x)=(F\cdot f)(x)=(T_{f}(F))(x)$$
 which implies that
$(T_{f}|_{A})'$ is the same as $T_{f}$ on $(A^{\sim})^{\sim}_{n}$.
Since $(T_{f}|_{A})'$ is interval preserving (\,cf. \cite [Theorem
7.8]{AB}), $T_{f}$ is likewise an interval preserving operator.
\end{proof}

\begin{corollary}\label{Corollary 2.2}
 For $f\in A^{\sim}, F\in (A^{\sim})^{\sim}_{n}$, we have $|F\cdot f|=|F|\cdot |f|$. Furthermore, if $f\perp g$ in
$A^{\sim}$, $F\cdot f\perp G\cdot g$ holds for any $F, G\in
(A^{\sim})^{\sim}_{n}$.
\end{corollary}

\begin{proof}
Since $V_{F}$ is an orthomorphism on $A^{\sim}$, we have
$V_{F}(f^{+})\perp V_{F}(f^{-})$ for each $f\in A^{\sim}$, that is,
$F\cdot(f^{+})\perp F\cdot(f^{-})$.   From Theorem \ref{Theorem 2.1}
we know
 \begin{eqnarray*}
|F\cdot f|&=&|F\cdot f^{+}|+|F\cdot
f^{-}|\\&=&|T_{f^{+}}(F)|+|T_{f^{-}}(F)|\\&=&T_{f^{+}}(|F|)+T_{f^{-}}(|F|)\\&=&|F|\cdot
f^{+}+|F|\cdot f^{-}=|F|\cdot |\,f|.
\end{eqnarray*}
Let $f\perp g$ in $A^{\sim}$. Then we have
\begin{eqnarray*}
|F\cdot f|\wedge |G\cdot g|&=&|F|\cdot |f|\wedge |G|\cdot
|\,g|\\&\leq&
 ((|F|+|G|)\cdot |f|)\wedge( (|F|+|G|)\cdot |\,g|)=0,
 \end{eqnarray*}
which implies that $F\cdot f\perp G\cdot g$ for all $F, G\in
(A^{\sim})^{\sim}_{\,n}$.
\end{proof}

\par Following the above discussion, we now consider $R(f)=\{F\cdot f: F\in(A^{\sim})^{\sim}_{\,n}\}$, the image
of $(A^{\sim})^{\sim}_{\,n}$ under $T_{f}$.
\begin{corollary}\label{Corollary 2.3}
 If $A$ is an $f$-algebra and $f\in (A^{\sim})$, then $R(f)=R(|\,f|)$, and $R(f)$ is an
order ideal in $A^{\sim}$.
\end{corollary}
\begin{proof}
First, since $T_{\,|f|}$ is an interval preserving lattice
homomorphism, we can easily see that $R(|\,f|)$ is an order ideal in
$A^{\sim}$. By Corollary \ref{Corollary 2.2} we conclude that
$R(f)\subseteq R(|\,f|)$.

 Now, to complete the proof we only need to prove that $R(|\,f|)\subseteq R(f)$. To this end, let $P_{\,1}: A^{\sim}\rightarrow
B_{f^{\,+}}$ , $P_{\,2}: A^{\sim}\rightarrow B_{f^{-}}$ be band
projections, where $B_{f^{\,+}}$ and $B_{f^{-}}$ are the bands
generated by $f^{\,+}$ and $f^{-}$ in $A^{\sim}$, respectively. If
$\pi=P_{\,1}-P_{\,2}$, we have $$\pi\in
\textrm{Orth}(A^{\sim}),\quad \pi(f)=|\,f|, \quad \pi(|\,f|)=f.$$In
addition, $\pi(f)\cdot a=\pi(f\cdot a)$ for all $a\in A$ (\,cf.
Theorem \ref{theorem 3.1}). Since $\pi$ is an orthomorphism on
$A^{\sim}$ and hence order continuous (cf. \cite [Theorem
8.10]{AB}), we have $\pi^{\prime}((A^{\sim})^{\sim}_{\,n})\subseteq
(A^{\sim})^{\sim}_{\,n}$. For all $a\in A$ and all
$F\in(A^{\sim})^{\sim}_{\,n}$, from
\begin{eqnarray*}(F\cdot
|\,f|)(a)&=&(F\cdot\pi(f))(a)\\&=&F(\pi(f)\cdot a)\\&=&F(\pi(f\cdot
a))\\&=&(\pi^{\prime}(F)\cdot f)(a)
\end{eqnarray*}it follows that $F\cdot|\,f|=\pi^{\prime}(F)\cdot f$
for all $F\in(A^{\sim})^{\sim}_{\,n}$, which implies that
$R(|\,f|)\subseteq R(f)$, as desired.
\end{proof}

\par Next we give a necessary and sufficient condition for $R(f)\perp R(g)$ when
$A$ has the factorization property. First we need the following
lemma.

\begin{lemma}
 Let $A$ be an $f$-algebra with the factorization property, and $f\in A^{\sim}$. If $f\cdot x=0$ for each $x\in A$,
then $f=0$.
\end{lemma}

\begin{proof}
Since $A$ has the factorization property, for each $a\in A^{+}$
there exist $x,y\in A$ such that $a=xy$. Hence from
$$f(a)=f(xy)=(f\cdot x)(y)=0$$  it follows easily that $f=0$ holds.
\end{proof}

\begin{theorem}\label{Theorem 2.5}
Let $A$ be an $f$-algebra with the factorization property. If
$f,g\in A^{\sim}$, then $f\perp g$ if and only if $R(f)\perp R(g)$.
\end{theorem}

\begin{proof}
 If $f\perp g$ in $A^{\sim}$, then it follows from Corollary \ref{Corollary 2.2} that $F\cdot f\perp G\cdot g$ for all
$F,G\in (A^{\sim})^{\sim}_{n}$ . This implies that $R(f)\perp R(g)$.

Conversely, if $R(f)$ and $R(g)$ are disjoint, then for each $F\in
((A^{\sim})^{\sim}_{\,n})^{+}$ we have
 \begin{eqnarray*}
 F\cdot(|f|\wedge|g|)&=&V_{F}(|f|\wedge|g|)\\&=&V_{F}(|f|)\wedge
V_{F}(|g|)\\&=&F\cdot |f|\wedge F\cdot|g|\\&=&|F\cdot f|\wedge
|F\cdot g|=0.
 \end{eqnarray*} In
particular, for any $x\in A$, its canonical image $x''\in
(A^{\sim})^{\sim}_{n}$ also satisfies $x''\cdot
(|f|\wedge|g|)=(|f|\wedge|g|)\cdot x=0$. By the preceding lemma we
have $|f|\wedge|g|=0$, i.e., $f\perp g$, as desired.
\end{proof}
Now we give the definition of the so-called $f$-orthomorphism.
\begin{definition}
 Let $A$ be an $f$-algebra and $T\in L_{b}(A^{\sim})$. $T$ is called an
$f$-orthomorphism on $A^{\sim}$ if $R(Tf)\subseteq R(f)$ for each
$f\in A^{\sim}$. The collection of all $f$-orthomorphisms on
$A^{\sim}$ will be denoted by $\textrm{Orth}(A^{\sim}, A^{\sim};
(A^{\sim})^{\sim}_{n})$.
\end{definition}
\par The next result deals with the relationship between the $f$-orthomorphisms  and the orthomorphisms on the order
dual of an $f$-algebra with the factorization property. Note that
$\textrm{Orth}(A^{\sim})$ is a band in $L_{b}(A^{\sim})$.
\begin{theorem}\label{theorem 2.7}
Let $A$ is an $f$-algebra. Then$\textrm{ Orth}(A^{\sim}, A^{\sim};
(A^{\sim})^{\sim}_{n})$ is a linear subspace of $ L_{b}(A^{\sim})$
and $\textrm{Orth}(A^{\sim})\subseteq  \textrm{Orth}(A^{\sim},
A^{\sim}; (A^{\sim})^{\sim}_{n})$.
 \par If $A$, in addition, has the factorization property, then $ Orth(A^{\sim}, A^{\sim}; (A^{\sim})^{\sim}_{n})= Orth(A^{\sim})$.
\end{theorem}
\begin{proof}
 First we can easily see that
$\textrm{Orth}(A^{\sim},A^{\sim};(A^{\sim})^{\sim}_{n})$ is a linear
subspace of $ L_{b}(A^{\sim})$. To prove
$\textrm{Orth}(A^{\sim})\subseteq\textrm{Orth}(A^{\sim}, A^{\sim};
(A^{\sim})^{\sim}_{n})$ let $\pi\in\textrm{Orth}(A^{\sim}). $ We
claim that $F\cdot\pi(f)=\pi\,'(F)\cdot f$ for all $F\in
(A^{\sim})^{\sim}_{n}$ and all $f\in A^{\sim}$. To this end, let
$F\in (A^{\sim})^{\sim}_{n}$, $f\in A^{\sim}$ and $x\in A$ be
arbitrary. Since
 $(A^{\sim})^{\sim}_{n}$ is a commutative $f$-algebra, by Theorem \ref{theorem 3.1} we have
\begin{eqnarray*}
(\pi\,'(F)\cdot f)(x)=\pi\,'(F)(f\cdot x)=F(\pi(f\cdot
x))&=&F(\pi(x''\cdot
f))\\&=&F(x''\cdot (\pi(f)))\\
&=&(F\cdot x'')(\pi(f))\\&=&(x''\cdot F)(\pi(f))\\&=&x''(F\cdot
\pi(f))=(F\cdot \pi(f))(x).
\end{eqnarray*}
 Thus,
$F\cdot\pi(f)=\pi\,'(F)\cdot f$. This implies that
$R(\pi(f))\subseteq R(f)$ for each $f\in A^{\sim}$, that is,
$\textrm{Orth}(A^{\sim})\subseteq\textrm{Orth}(A^{\sim}, A^{\sim};
(A^{\sim})^{\sim}_{n})$.
\par If $A$ has the factorization property, we prove that $ \textrm{Orth}(A^{\sim},
A^{\sim};(A^{\sim})^{\sim}_{n})\subseteq \textrm{Orth}(A^{\sim})$
holds. To this end,  take $T\in\textrm{ Orth}(A^{\sim},
A^{\sim};(A^{\sim})^{\sim}_{n})$ and $f,\,g\in A^{\sim}$ saisfying
$f\perp g$ in $A^{\sim}$. Then it follows from Theorem \ref{Theorem
2.5}
 that $R(f)\perp R(g)$. Since $T\in \textrm{Orth}(A^{\sim},A^{\sim};(A^{\sim})^{\sim}_{n})$, we
have $R(T(f))\subset R(f)$. Therefore $R(T(f))\perp R(g)$, which
implies that $T(f)\perp g$, and hence $T$ is an orthomorphism on
$A^{\sim}$, as desired.
\end{proof}

\section{$f$-linear operators on the order dual}
 Let $A$ be an $f$-algebra with separating order dual and $T\in L_{b}(A^{\sim})$. Recall that $T$ is called to be \textit{$f$-linear} with respect to $(A^{\sim})^{\sim}_{n}$
 if $T(G\cdot f)=G\cdot T(f)$ for all $f\in A^{\sim}$ and $G\in
(A^{\sim})^{\sim}_{n}$. The set of all $f$-linear operators on
$A^{\sim}$ will be denoted by
$L_{b}(A^{\sim},A^{\sim};(A^{\sim})^{\sim}_{n})$. It follows from
\cite[Lemma 4.4]{LP} that
$L_{b}(A^{\sim},A^{\sim};(A^{\sim})^{\sim}_{n})$ is a band in
$L_{b}(A^{\sim})$.

\begin{theorem}\label{theorem 3.1}
  Let $A$ be an $f$-algebra with separating order dual. Then $Orth(A^{\sim})\subseteq L_{b}(A^{\sim},A^{\sim};(A^{\sim})^{\sim}_{n})$.
\end{theorem}
\begin{proof}
 Clearly $\textrm{Orth}(A^{\sim})$ is commutative since $\textrm{Orth}(A^{\sim})$ is an Archimedean $f$-algebra.
 To complete the proof, let $\pi\in Orth(A^{\sim})$.  We have $$\pi(G\cdot f)=\pi(V_{G}(f))=V_{G}(\pi(f))=G\cdot
(\pi(f))$$for all $f\in A^{\sim}$ and $G\in (A^{\sim})^{\sim}_{n}$.
Hence $\pi\in L_{b}(A^{\sim},A^{\sim};(A^{\sim})^{\sim}_{n})$ .
\end{proof}
\par The following result deals with the order adjoint of an $f$-linear operator on the order dual of an $f$-algebra.
It should be noted that the order adjoint of an order bounded
operator is order continuous (cf. \cite[Theorem 5.8]{AB}).
\begin{lemma}\label{lemma 3.2}
 Let $T\in L_{b}(A^{\sim},A^{\sim};(A^{\sim})^{\sim}_{n})$. Then
the order adjoint \,$T\,'$ of $T$ satisfies $T\,'(F)\cdot f=F\cdot
T(f)$ for all $F\in (A^{\sim})^{\sim}_{n}$ and $f\in A^{\sim}$. In
particular, $G\cdot T\,'(F)=T\,'(G\cdot F)$ for all $F,G \in
(A^{\sim})^{\sim}_{n}$.
\end{lemma}

\begin{proof}
 Since $T\in L_{b}(A^{\sim},A^{\sim};(A^{\sim})^{\sim}_{n})$, and $(A^{\sim})^{\sim}_{n}$ is a
commutative $f$-algebra, we have
\begin{eqnarray*}
(T\,'(F)\cdot f)(x)=T\,'(F)(f\cdot x)&=&F(T(f\cdot
x))\\&=&F(T(x''\cdot
f))\\&=&F(x''\cdot (T(f)))\\
&=&(F\cdot x'')(T(f))\\&=&(x''\cdot F)(T(f))\\&=&x''(F\cdot
T(f))=(F\cdot T(f))(x)
\end{eqnarray*}
 for all $F\in
(A^{\sim})^{\sim}_{n}$, $f\in A^{\sim}$ and $x\in A$, which implies
that $T\,'(F)\cdot f=F\cdot T(f)$.
\par Let $F,\,G\in (A^{\sim})^{\sim}_{n}$ be given. Then for $f\in
A^{\sim}$, from
\begin{eqnarray*}
(G\cdot T\,'(F))(f)=G( T\,'(F)\cdot f)&=&G(F\cdot T(f))\\&=&(G\cdot
F)(T(f))\\&=&(T\,'(G\cdot F))(f)
\end{eqnarray*}
 it follows that $G\cdot T\,'(F)=T\,'(G\cdot F)$. This completes the
 proof.
\end{proof}

\begin{theorem}\label{theorem 3.3}

$L_{b}(A^{\sim},A^{\sim};(A^{\sim})^{\sim}_{n})\subseteq
Orth(A^{\sim},A^{\sim};(A^{\sim})^{\sim}_{n})$.
\end{theorem}
\begin{proof}
 For $T\in L_{b}(A^{\sim},A^{\sim};(A^{\sim})^{\sim}_{n}) $, we know that $|\,T|$ is also $f$-linear with
 respect to $(A^{\sim})^{\sim}_{n}$. Assume that $0\leq G\in(A^{\sim})^{\sim}_{n}$ and $f\in A^{\sim}$. So by Lemma \ref{lemma 3.2},
we have$$0\leq G\cdot (|T(f)|)\leq G\cdot (|T||\,f|)=(|T|'(G))\cdot
|f|=T_{|\,f|}(|T|'(G))$$Since $T_{|\,f|}$ is interval preserving,
there exists $F\in (A^{\sim})^{\sim}_{n} $ such that $0\leq F\leq
|T|'(G)$
 and $G\cdot (|T(f)|)=F\cdot |f|$. It is now immediate that $R(|T(f)|)\subseteq
 R(|f|)$ and hence
 $L_{b}(A^{\sim},A^{\sim};(A^{\sim})^{\sim}_{n})\subseteq
\textrm{Orth}(A^{\sim},A^{\sim};(A^{\sim})^{\sim}_{n})$, as desired.
 \end{proof}

\par Combining Theorem \ref{theorem 3.1}, Theorem \ref{theorem 3.3} with Theorem \ref{theorem 2.7}, we have the
following result.

\begin{theorem}\label{theorem 3.4}
 If $A$ is an $f$-algebra with separating order dual, then
 $$Orth(A^{\sim})\subseteq
L_{b}(A^{\sim},A^{\sim};(A^{\sim})^{\sim}_{n})\subseteq
Orth(A^{\sim},A^{\sim};(A^{\sim})^{\sim}_{n}).$$In particular, if,
in addition, $A$ has the factorization property,
then$$Orth(A^{\sim})=
L_{b}(A^{\sim},A^{\sim};(A^{\sim})^{\sim}_{n})=
Orth(A^{\sim},A^{\sim};(A^{\sim})^{\sim}_{n}).$$
\end{theorem}
\section*{Acknowledgement}
The authors would like to thank the reviewer for his/her kind
comments and valuable suggestions which have improved this paper.
Corollary \ref{Corollary 2.3} in its present formulation and the
proof are essentially due to the reviewer. In particular, he/she
suggested that the authors should consider their questions under the
condition of ``$A$ has the factorization property ", which is weaker
than the hypothesis ``$A$ is square-root closed" used originally in
this paper.

\end{document}